\documentclass[
submission
]{dmtcs-episciences}

\usepackage[utf8]{inputenc}
\usepackage{subfigure}

\usepackage[round]{natbib}
\usepackage{amsmath,amssymb,amsfonts}
\usepackage{algorithmic}
\usepackage{graphicx}
\usepackage{textcomp}
\usepackage{float}

\newtheorem{lemma}{Lemma}[section]

\newtheorem{theorem}[lemma]{Theorem}

\author{Enqiang Zhu\affiliationmark{1}
  \and Chanjuan Liu\affiliationmark{2}
  \and Yongsheng Rao\affiliationmark{1}\thanks{Corresponding author: ysrao2018@163.com}
  }
\title[On a Sufficient Condition for Planar Graphs of Maximum Degree 6 to be Totally 7-Colorable]{On a Sufficient Condition for Planar Graphs of Maximum Degree 6 to be Totally 7-Colorable}
\affiliation{
  Institute of Computing Science and Technology,  Guangzhou University, Guangzhou 510006, China\\
  School of Computer Science and Technology, Dalian University of
Technology,  Dalian, China}

\begin{document}
%\publicationdetails{VOL}{2015}{ISS}{NUM}{SUBM}
\maketitle
\begin{abstract}
  A total $k$-coloring of a graph is an assignment of $k$ colors to its vertices and edges such that no two adjacent or incident elements receive the same color. The Total Coloring Conjecture (TCC) states that every simple graph $G$ has a total ($\Delta(G)+2$)-coloring, where $\Delta(G)$ is the maximum degree of $G$. This conjecture has been confirmed for planar graphs with maximum degree at least 7 or at most 5, i.e., the  only open
case of TCC is that of maximum degree 6. It is known that every planar graph $G$ of $\Delta(G) \geq 9$ or
$\Delta(G) \in \{7, 8\}$ with some restrictions has a total $(\Delta(G) + 1)$-coloring. In particular, in [Shen and Wang, ``On the 7 total colorability of planar graphs with maximum degree 6 and without 4-cycles", Graphs and Combinatorics, 25: 401-407, 2009], the authors proved that every planar graph with maximum degree 6 and without 4-cycles has a total 7-coloring. In this paper, we improve this result by showing that every diamond-free and house-free planar graph of maximum degree 6 is totally 7-colorable if  every 6-vertex is not incident with two
adjacent 4-cycles or not incident with three cycles of size $p,q,\ell$ for some $\{p,q,\ell\}\in \{\{3,4,4\},\{3,3,4\}\}$.
\end{abstract}

\section{Introduction}
\label{sec:introduction}
Throughout the paper, we consider only simple, finite and undirected planar graphs, and follow \cite{Bondy2008} for terminologies and notations not defined here. Given a graph $G$, we use $V(G)$ and  $E(G)$  to denote the \emph{vertex set} and the \emph{edge set} of $G$, respectively. For a vertex $v\in V(G)$,  we denote by $d_G(v)$ the \emph{degree} of $v$ in $G$ and let $N_G(v)=\{u|uv \in E(G)\}$. A $k$-vertex, $k^-$-vertex or $k^+$-vertex is a vertex of degree $k$, at most $k$ or at least $k$. For a planar graph $G$, we always assume that $G$ is embedded in the plane, and denote by $F(G)$ the set of faces of $G$. The degree of a face $f\in F(G)$, denoted by $d_G(f)$, is the number of edges incident with $f$, where each cut-edge is counted twice. A face of degree $k$, at least $k$ or at most $k$ is called a \emph{$k$-face, $k^+$-face, or $k^-$-face}. A $k$-face with consecutive vertices $v_1,v_2,\ldots,v_k$ along its boundary in some direction is often said to be a $(d_G(v_1), d_G(v_2),\ldots,d_G(v_k))$-$face$. Two faces are called \emph{adjacent} if they are incident with  a common edge.

A \emph{total $k$-coloring} of a graph $G$ is a coloring from  $V(G)\cup E(G)$ to $\{1,2,\ldots,k\}$  such that no two adjacent or incident elements have the same color. A graph  $G$ is said to be totally \emph{$k$-colorable} if it admits a total $k$-coloring. The \emph{total chromatic number} of $G$, denoted by  $\chi_t(G)$, is the smallest integer $k$ such that $G$ is totally $k$-colorable.  The Total Coloring Conjecture (TCC) states that every simple graph $G$ is totally ($\Delta(G)+2$)-colorable \cite{Behzad1965,Vizing1968}, where $\Delta(G)$ is the maximum degree of $G$.
 This conjecture  has been proved for graphs $G$ with $\Delta(G)\leq 5$  in \cite{Kostochka1996}. For planar graphs, the only open case of TCC is that of maximum degree 6;  see \cite{Borodin1989,Jensen1995,Sanders1999}.  More precisely, if $G$ is a planar graph with $\Delta\geq 9$, then $\chi_t(G)=\Delta(G)+1$. For planar graphs  with maximum degree 7 or 8, some related results can be found in \cite{Chang2011,Chang2013,Du2009,Hou2011,Hou2008,Liu2009,Shen2009,Wang2011,Wang2012a,Xu2014,Wang2014}. Moreover, for planar graphs $G$ of maximum degree 6,  it is proved that $\chi_t(G)=7$ if $G$ does not contain 5-cycles \cite{Hou2011} or 4-cycles \cite{Shen2009a}. In this paper, we show that every planar graph $G$ with $\Delta(G)=6$ has a total 7-coloring if $G$ contains no some forbidden 4-cycles, which improves the result of \cite{Shen2009a}.

\begin{figure}[H]
\centering
  \includegraphics[width=150pt]{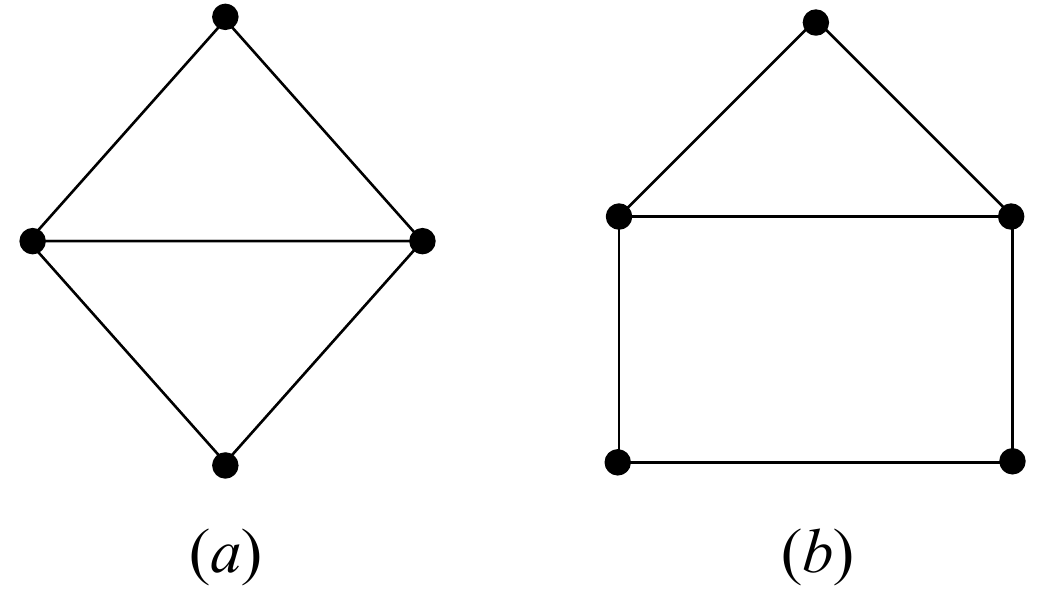}\\
%(a)\hspace{3.3cm}(b)
  \caption{(a) diamond (b) house}\label{figure0}
\end{figure}

\begin{theorem}\label{mainresult}
Suppose that  $G$ is  a planar graph with $\Delta(G)=6$. If $G$ does not contain a subgraph isomorphic to a diamond or a house, as shown in Figure \ref{figure0}, and every 6-vertex in $G$ is not incident with two adjacent 4-cycles or three cycles with sizes $p,q,\ell$ for some  $\{p,q,\ell\}\in \{\{3,4,4\},\{3,3,4\}\}$, then $\chi_t(G)=7$.
\end{theorem}

\section{Reducible configurations}
\label{sec2}
Let $H$ be a  minimal counterexample to Theorem \ref{mainresult}, in the sense that the quantity $|V(H)|+|E(H)|$ is minimum.  That is, $H$ satisfies the following properties:

(1) $H$ is a planar graph of maximum degree  6.

(2) $H$ contains no subgraphs isomorphic to a diamond or a house.

(3) Every 6-vertex  of $H$ is incident with neither two adjacent 4-cycles, nor  three cycles with sizes $p,q,l$ for some $\{p,q,\ell\}\in \{\{3,4,4\},\{3,3,4\}\}$.

(4) $H$ is not totally 7-colorable such that $|V(H)|+|E(H)|$ is minimum subject to (1),(2),(3).

Notice that every planar graph with maximum degree 5  is totally 7-colorable \cite{Kostochka1996}. Additionally, it is easy to check that every subgraph of $H$ also possesses (2) and (3). Therefore, every proper subgraph of $H$ has a total 7-coloring $\phi$ using the color set $C$=$\{1,2,\ldots, 7\}$. For a vertex $v$, we use  $C_{\phi}(v)$ to denote the set of colors appearing on $v$ and its incident edges, and $\overline{C}_{\phi}(v)$=$\{1,2,\ldots,7\}\setminus C_{\phi}(v)$.
This section is devoted to investigating some structural information, which shows that certain configurations are \emph{reducible}, i.e. they can not occur in $H$.

\begin{lemma}\label{lemma1}
 $(1)$ Let $uv$ be an edge of $H$ such that $d_H(u)\leq 3$ or $d_H(v)\leq 3$. Then $d_H(u)+d_H(v)\geq 8$.

 $(2)$ The subgraph induced by all edges, whose two ends are 2-vertex and 6-vertex respectively in $H$ is a forest.
\end{lemma}

The proof of Lemma \ref{lemma1} can be found in \cite{Borodin1989}.

For any component $T$ of the forest stated in Lemma \ref{lemma1} (2),  we can see that all leaves (i.e. 1-vertices) of $T$ are 6-vertices. Therefore, $T$ has a maximum matching $M$ that saturates  every 2-vertex in $T$. For each 2-vertex $v$ in $T$, we refer to the neighbor of $v$ that is saturated by $M$ as the \emph{master} of $v$, see \cite{Borodin1997}. Clearly, for a given $M$, each 6-vertex can be the master of at most one 2-vertex, and each 2-vertex has exactly one master.

The following result follows from Lemma \ref{lemma1} directly.

\begin{lemma}\label{old3}
Every 4-face in $H$ is incident with at most one 2-vertex.
\end{lemma}

\begin{lemma}\label{lemma2}
Let $f$ be a 3-face incident with a 2-vertex. Then every 6-vertex incident with $f$ has only one neighbor of degree 2.
%Let $v_1v_2v_3$ be a (2,6,6)-face of $H$, where $d_H(v_1)=2, d_H(v_2)=d_H(v_3)=6$. Then, $v_2$ (resp. $v_3$) has no 2-neighbor different from $v_1$, and at most two 3-neighbors.
\end{lemma}
\begin{proof}
Let $v_1$ be the 2-vertex incident with $f$, and $v_2,v_3$ be the two 6-vertices incident with $f$. We first show that the result holds for $v_2$, and then holds for $v_3$ analogously.  Assume to the contrary  that $v_2$ has another neighbor of degree 2, say $u (\neq v_1)$. Let $\phi$ be a total 7-coloring of $H-v_1v_2$ by the minimality of $H$. Erase the colors on $v_1$ and $u$.   Without loss of generality, we assume $\overline{C}_{\phi}(v_2)$=$\{7\}$. If $\phi(v_1v_3)\neq 7$, then $v_1v_2$ can be properly colored with 7.  Hence, $H$ has a total  7-coloring by coloring $v_1,u$ properly (Since  $v_1,u$ are 2-vertices, there are at least three available colors for each of them), and a contradiction. So we assume $\phi(v_1v_3)=7$. Let $w=N_H(u)\setminus \{v_2\}$. When $\phi(uw)\neq 7$, we can  color $v_1v_2$ with $\phi(v_2u)$ and recolor $v_2u$ with 7. When $\phi(uw)=7$, let $\phi(v_2v_3)=x$ and  $\phi(v_2u)=y$. We  first exchange the colors of $v_1v_3$ and $v_2v_3$, and then color $v_1v_2$ with $y$ and recolor $v_2u$ with $x$. Therefore, we can obtain a 7-total-coloring of $H$ by coloring $v_1,u$ with two available colors. This contradicts the assumption of $H$.
\end{proof}

\begin{lemma}\label{lemma444}
$H$ has no (4,4,4)-face.
\end{lemma}

\begin{proof}
Suppose that $H$ has a (4,4,4)-face with three incident vertices $v_1$, $v_2$ and $v_3$. By the minimality of $H$, $H-\{v_1v_2,v_2v_3,v_3v_1\}$ has a total 7-coloring. Erase the colors on $v_i$ for $i=1,2,3$. Clearly, each element in $\{v_1,v_2,v_3,v_1v_2,v_2v_3,v_3v_1\}$ has at least three available colors. Since every 3-cycle is totally 3-choosable, it follows  that $H$ has a total 7-coloring, and a contradiction.
\end{proof}

\begin{lemma}\label{lemma4}
$H$ has no (3,5,3,5)-face.
\end{lemma}
\begin{proof}
Assume to the contrary  that $H$ has a (3,5,3,5)-face $f=v_1v_2v_3v_4$, where $d_H(v_1)=d_H(v_3)=3$,  $d_H(v_2)=d_H(v_4)=5$. By the minimality of $H$, $H-\{v_1v_2,v_2v_3,v_3v_4,v_4v_1\}$ has a total 7-coloring $\phi$. Erase the colors on $v_1,v_3$. Clearly, each edge of $\{v_1v_2,v_2v_3,v_3v_4,v_4v_1\}$ has at least two available colors. Since even cycles are 2-edge-choosable, we can properly color edges $v_1v_2,v_2v_3,v_3v_4$ and $v_4v_1$. Additionally, since $v_1$, $v_3$ are 3-vertices and they are not adjacent (Since $H$ contains no subgraph isomorphic to a diamond), we can properly color $v_1$ and $v_3$ with two available colors. Hence, we obtain a total 7-coloring of $H$, and a contradiction.
\end{proof}

\begin{lemma}\label{lemmaadd}
Every 6-vertex incident with a 2-vertex in $H$ is adjacent to at most five $3^-$-vertices.
\end{lemma}
\begin{proof}
Let $v$ be a 6-vertex incident with a 2-vertex $v_1$ in $H$. Assume to the contrary that $N_H(v)$ contains six $3^-$-vertices. Let $N_H(v)= \{v_1,v_2,v_3,v_4,v_5,v_6\}$, where $d_H(v_1)=2,d_H(v_i)\leq 3$ for $i=2,3,4,5,6$. By the minimality of $H$, $H-\{vv_1\}$ has a total 7-coloring $\phi$. Without loss of generality, we assume $\overline{C}_{\phi}(v)=\{7\}$. Erase the colors on $v_i$ for $i=1,2,3,4,5,6$. If 7 does not appear on the edges incident with $v_1$,  then we can properly color  $vv_1$ with 7. Otherwise, we can properly color $vv_1$ with $\phi(v)$ by recoloring $v$ with  7. Additionally, since $v_1,v_2,v_3,v_4,v_5,v_6$ are $3^-$-vertices,  there is at least one available color for each of them and  by Lemma \ref{lemma1} (1) $v_iv_j\notin E(H)$ for any $i,j\in \{1,2, \ldots, 6\}$, $i\neq j$. Hence, we can obtain a 7-total-coloring of $H$, and a contradiction.
\end{proof}

\begin{lemma}\label{lemma-add}
$H$ contains no configurations depicted in Figure \ref{figureadd}, where the vertices marked by $\bullet$ have no other neighbors in $H$.
\end{lemma}

%\Figure[t!](topskip=0pt, botskip=0pt, midskip=0pt)[width=7cm]{fig12.pdf}
%{Three forbidden configurations in $H$ \label{figureadd}}

\begin{figure}[H]
\centering
  \includegraphics[width=200pt]{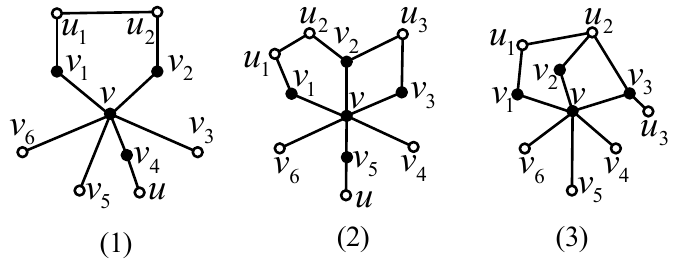}\hspace{0.1cm}
\caption{ Three forbidden configurations in $H$}\label{figureadd}
\end{figure}

\begin{proof}
%Since the proof for these three configurations are similar, we here only give the proof for configuration (3).
For configuration (1), by the minimality of $H$, $H-\{vv_1\}$ has a 7-total-coloring $\phi$. Without loss of generality, we assume that $C_\phi(v)=\{1,2,3,4,5,6\}$. If $\phi(v_1u_1)\neq7$ (or $\phi(v_2u_2)\neq7$), then we can properly color $vv_1$ with 7 (or with $\phi(vv_2)$ by recoloring $vv_2$ with 7). If $\phi(v_4u)\neq7$, then we can properly color $vv_1$ with  $\phi(vv_4)$ by recoloring $vv_4$ with 7.
So, we assume $\phi(u_1v_1)=\phi(u_2v_2)=\phi(uv_4)=7$. Let $\phi(u_1u_2)=c_1$. Obviously, $c_1\neq 7$. If $c_1\neq \phi(vv_2)$, then we can recolor $u_1u_2$ with 7, $u_1v_1$ with $c_1$, $u_2v_2$ with $c_1$, and then properly color $vv_1$ with 7. If $c_1=\phi(vv_2)$, then $c_1\neq \phi(vv_4)$. Therefore, we can safely interchange the colors of $vv_2$ and $vv_4$, recolor $u_1u_2$ with 7, $u_1v_1$ with $c_1$, $u_2v_2$ with $c_1$,  and then properly color $vv_1$ with 7. Thus, we obtain a 7-total-coloring of $H$, and a contradiction.

For configuration  (2), let $\phi$ be a 7-total-coloring of $H-\{vv_1\}$. Assume that $C_\phi(v)=\{1,2,3,4,5,6\}$, where $\phi(vv_i)=i-1$ for $i=2,3,4,5,6$, and $\phi(v)=6$. 
By a similar argument as in (1), we assume $\phi(u_1v_1)=\phi(u_2v_2)=\phi(u_3v_3)=\phi(uv_5)=7$. Let $\phi(u_1u_2)=c_1$ and $\phi(v_2u_3)=c_2$.
 Obviously, $c_1,c_2\neq 7$.
 First, if $c_1\notin \{1, c_2\}$, then we can recolor $u_1u_2$ with 7, $u_1v_1$ with $c_1$, $u_2v_2$ with $c_1$, and then properly color $vv_1$ with 7. Second, if $c_1=1$, then $c_1\neq c_2$. When $c_2\neq 4$, we  can safely interchange the colors of $vv_2$ and $vv_5$, recolor $u_1u_2$ with 7, $u_1v_1$ with $c_1$, $u_2v_2$ with $c_1$,  and then properly color $vv_1$ with 7. When $c_2=4$, we  can safely interchange the colors of $vv_2$ and $vv_3$, recolor $u_1u_2$ with 7, $u_1v_1$ with $c_1$, $u_2v_2$ with $c_1$,  and then properly color $vv_1$ with 7.
 Third, if $c_1=c_2$, then $c_1\neq 1$. When $c_1\neq 2$, we can recolor $u_1u_2$ and $v_2u_3$ with 7, recolor $u_1v_1$, $u_2v_2$ and $u_3v_3$ with $c_1$, and then properly color $vv_1$ with 7. When $c_1=2$, we  can safely interchange the colors of $vv_3$ and $vv_5$,  recolor $u_1u_2$ and $v_2u_3$ with 7, recolor $u_1v_1$, $u_2v_2$ and $u_3v_3$ with $c_1$, and then properly color $vv_1$ with 7.
 Hence, we obtain a 7-total-coloring of $H$, and a contradiction.

For configuration  (3), let $\phi$ be a 7-total-coloring of $H-\{vv_1\}$. Assume that $C_\phi(v)=\{1,2,3,4,5,6\}$, where $\phi(vv_i)=i-1$ for $i=2,3,4,5,6$, and $\phi(v)=6$.  
By a similar argument as in (1), assume that $\phi(u_1v_1)=\phi(u_2v_2)=\phi(u_3v_3)=7$. Let $\phi(u_1u_2)=c_1$ and $\phi(u_2v_3)=c_2$. Obviously, $c_1,c_2\neq 7$ and $c_1\neq c_2$. If $c_1\neq 1$, then we can recolor $u_1u_2$ with 7, $u_1v_1$ with $c_1$, $u_2v_2$ with $c_1$, and then properly color $vv_1$ with 7. If $c_1=1$, then $c_2\neq 1$. Therefore, we can recolor $u_1u_2$ with 7, $u_1v_1$ with $c_1$, $u_2v_2$ with $c_2$,  $u_2v_3$ with $c_1$,  and then properly color $vv_1$ with 7. So, we obtain a 7-total-coloring of $H$, and a contradiction.
\end{proof}

\section{Discharging}
\label{sec3}

In this section, to complete the proof of Theorem \ref{mainresult}, we will use discharging method to derive a contradiction. For a vertex $v$, denote by $n_3(v)$ and $n_4(v)$ (or simply by $n_3$ and $n_4$) respectively the number of 3-faces and 4-faces incident with $v$. For a face $f$, denote by $m_2(f)$ and $m_3(f)$ (or simply by $m_2$ and $m_3$)  respectively the number of 2-vertices and 3-vertices incident with $f$.

According to Euler's formula $|V(H)|-|E(H)|+|F(H)|=2$, we have

\[
\sum\limits_{v\in V(H)}(d_H(v)-4)+\sum\limits_{f\in F(H)}(d_H(f)-4)=-8<0.
\]

Now, we define $c(x)$ to be the \emph{initial charge} of $x\in V(H)\cup F(H)$. Let $c(x)=d_H(x)-4$ for each $x\in V(H)\cup F(H)$. Obviously, $\sum_{x\in (V(H)\cup F(H))}ch (x)=-8<0$. Then, we apply the following rules to reassign the initial charge that leads to a new charge $c'(x)$. If we can show that $c'(x)\geq 0$ for each $x\in V(H)\cup F(H)$, then we obtain a contradiction, and complete the proof.

\begin{description}
\item[(R1)]  From each $k$-vertex to each of its incident $k'$-face $f$, transfer

    $\frac{1}{3}$, if $k\in\{5,6\}$, $k'=3$ and $f$ is a ($5^+, 5^+, 5^+$)-face;

    $\frac{1}{2}$, if $k\in\{5,6\}$, $k'=3$ and $f$ is a ($4^-, 5^+, 5^+$)-face or ($4^-, 4^-, 5^+$)-face;

     $\frac{1}{5}$, if $k=5$, $k'=4$ and $f$ is incident with a 2-vertex or 3-vertex.

\item[(R2)] From each 6-vertex to each of its incident 4-face $f$, transfer

     $\frac{1}{4}$, if $f$ is a ($2, 6, 4^+,6$)-face;

     $\frac{5}{12}$, if $f$ is a ($2, 6, 3,6$)-face;

       $\frac{2}{15}$, if $f$ is a ($3, 5^+, 5^+, 5^+$)-face or ($3, 6, 4, 5$)-face

        $\frac{1}{6}$, if $f$ is a ($3, 6, 4, 6$)-face;

       $\frac{1}{3}$, if $f$ is a ($3, 6, 3, 6$)-face;

       $\frac{7}{15}$, if $f$ is a ($3, 6, 3,5$)-face.

\item[(R3)] From each 6-vertex $u$  to each of its adjacent 2-vertex $v$, transfer

   $\frac{1}{2}$, if $v$ is incident with a 3-face;

   $\frac{4}{5}$, if $v$ is not incident with a 3-face and $u$ is a master of $v$;

    $\frac{1}{5}$, if $v$ is not incident with a 3-face and $u$ is not a master of $v$.

\item[(R4)] From each 4-face  to each of its adjacent $k$-vertex $v$, transfer

   $\frac{1}{2}$, if $k=2$ and $v$ is not incident with a 3-face;

   $\frac{1}{3}$, if $k=3$ and $v$ is not incident with a 3-face.

\item[(R5)] From each $5^+$-face  to each of its adjacent $k$-vertex $v$, transfer

   $1$, if $k=2$ and $v$ is incident with a 3-face;

   $\frac{1}{2}$, if $k=2$ and $v$ is not incident with a 3-face;

   $\frac{1}{2}$, if $k=3$ and $v$ is  incident with a 3-face;

   $\frac{1}{3}$, if $k=3$ and $v$ is not incident with a 3-face.

 \item[(R6)] Each $5^+$-face transfer $\frac{1}{6}$ to its adjacent ($4^-,4^-,5^+$)-face.
\item[(R7)] Every $4^+$-face with positive charge after R1 to R6 transfers its remaining charges evenly among its incident 6-vertices.
\end{description}

The rest of this article is to check that $c'(x)\geq 0$ for every $x\in V(H)\cup F(H)$.

\section{Final charge of faces}
\label{sec4}
Let $f$ be a face of $H$.  Suppose that $f$ is a 3-face. By Lemma \ref{lemma1}(1) and Lemma \ref{lemma444}, it follows that $f$ is incident with at most two $4^-$-vertices. If $f$ is incident with at most one $4^-$-vertex, then by (R1) $c'(f)=3-4+3\times \frac{1}{3}=0$ or $c'(f)=3-4+2\times \frac{1}{2}=0$. If $f$ is incident with two $4^-$-vertices, then by (R1) and (R6), $c'(f)=3-4+\frac{1}{2}+3\times\frac{1}{6}=0$.

Suppose that $f$ is a 4-face. Clearly, $f$ is not adjacent to a 3-face, since $H$ does not contain any subgraph isomorphic to a house.  If $f$ is incident with neither a 2-vertex nor a 3-vertex, then $c'(f)=c(f)=0$; If $f$ is incident with a 2-vertex, then $f$ is a (2,6,$3^+$,6)-face by Lemma \ref{lemma1} (1) and Lemma \ref{old3}, and the 2-vertex is not incident with any 3-face (Since $H$ contains no subgraph isomorphic to a diamond). So, by (R2), (R4) and (R7), $c'(f)=2\times \frac{5}{12}-\frac{1}{2}-\frac{1}{3}=0$ (When $f$ is incident with a 3-vertex), or $c'(f)=2\times \frac{1}{4}-\frac{1}{2}=0$ (When $f$ is not incident with any 3-vertex); If $f$ is not incident with a 2-vertex but is incident with a 3-vertex, then $f$ is either a (3,$5^+$,$4^+$,$5^+$)-face, or a  (3,$5^+$,$3$,$6$)-face by Lemma \ref{lemma1} (1) and Lemma \ref{lemma4}. For the former case, after (R1), (R2) and (R4), $f$ has at least $3\times \frac{2}{15}-\frac{1}{3}=\frac{1}{15}$ (When $f$ is (3,$5^+,5^+,5^+$)-face), or $\frac{1}{5}+\frac{2}{15}-\frac{1}{3}=0$  (When $f$ is ($3,5^+,4,5$)-face), or $2\times \frac{1}{6}-\frac{1}{3}=0$ (When $f$ is (3,6,4,6)-face). Therefore, $c'(f)\geq 0$ by (R7). For the latter case, $c'(f)=\frac{1}{5}+\frac{7}{15}-2\times\frac{1}{3}=0$ by (R1), (R2), (R4) and (R7) (When $f$ is (3,$5,3,6$)-face), or $c'(f)=\frac{1}{3}+\frac{1}{3}-2\times\frac{1}{3}=0$ by (R2), (R4) and (R7) (When $f$ is (3,$6,3,6$)-face).

Suppose that $f$ is a 5-face.  Since $H$ does not contain any subgraph isomorphic to a house, it has that every 2-vertex incident with it is not incident with a 3-face. Obviously, $m_2+m_3\leq 2$.  If $m_2+m_3\leq 1$, then $f$ has at least $5-4-\frac{1}{2}(m_2+m_3)-\frac{1}{6}(5-2(m_2+m_3))\geq 0$ after $(R5)$ to (R6), and hence $c'(f)\geq 0$ by $R7$. If $m_2+m_3=2$, then $f$ is not adjacent to any ($4^-,4^-,5^+$)-face. Hence, $f$ has at least $5-4-\frac{1}{2}(m_2+m_3)=0$ after $(R5)$ to (R6), and $c'(f)\geq 0$.

Suppose that $f$ is a 6-face. Then, at most one 2-vertex incident with $f$ is incident with a 3-face. Otherwise, $H$  contains a subgraph isomorphic to a house.
By Lemma \ref{lemma1} (1) and (2), it is easy to see that $m_2\leq 2$ and $m_2+m_3\leq 3$. When   $m_2+m_3\leq 2$, the number of $(4^-,4^-,5^+)$-faces adjacent to $f$ is at most $6-2(m_2+m_3)$. Therefore, $f$ has at least $\min\{6-4-1-\frac{1}{2}-2\times \frac{1}{6}, 6-4-1-4\times \frac{1}{6}, 6-4-6\times \frac{1}{6}\}=\frac{1}{6}$ after $(R5)$ to (R6), and hence $c'(f)\geq 0$ by $R7$. When $m_2+m_3=3$, it follows that $f$ is not adjacent to any $(4^-,4^-,5^+)$-faces. Therefore, $f$ has at least $6-4-1-2\times \frac{1}{2}=0$ after $(R5)$ to (R6), and $c'(f)\geq 0$.

For the convenience of proving $c'(v)\geq 0$ for every $v\in V(H)$, we  first introduce the following Lemma, which indicates that every $7^+$-face has positive charges.

\begin{lemma} \label{chargeto6-vertex}
Let $v$ be a 6-vertex.  Then $v$ receives at least $\frac{1}{8}$ from each of its incident $7^+$-face by $(R7)$.
\end{lemma}
\begin{proof}
Let $f$ be a $k$ ($\geq 7$)-face incident with $v$. Clearly, the number ($4^-,4^-,5^+$)-faces adjacent to $f$ is at most $k-2(m_2+m_3)$.

Suppose $k\geq 8$. Then  $f$ has at least $k-4-m_2-\frac{1}{2}\times m_3-\frac{1}{6}\times (k-2\times(m_2+m_3))=\frac{5}{6}k-4-\frac{2}{3}m_2-\frac{1}{6}m_3$ charges after (R5) to (R6). Since $0 \leq m_2\leq \lfloor\frac{k-1}{2}\rfloor$ by Lemma \ref{lemma1} (2) and $ m_2+m_3\leq \lfloor\frac{k}{2}\rfloor$ by Lemma \ref{lemma1} (1). Therefore, $v$ receives at least  $\frac{\frac{5}{6}k-4-\frac{2}{3}m_2-\frac{1}{6}m_3}{k-m_2-m_3}=$ $\frac{5}{6}-\frac{1}{6}\cdot \frac{24-m_2-4m_3}{k-m_2-m_3}\geq \frac{5}{6}-\frac{1}{6}\times \frac{17}{4}=\frac{1}{8}$ (when $m_2=3, m_3=1,k=8$) from $f$.

Suppose $k=7$. Clearly, $m_2+m_3\leq 3$. Particularly, in the case of $m_2+m_3=3$, $f$ is not adjacent to any ($4^-,4^-,5^+$)-face. First, when $m_2=3$, it has that $f$ is incident with at most two 2-vertices that are incident with a 3-face (Otherwise, there is a subgraph isomorphic to a house, and a contradiction). Therefore, $f$ has at least $7-4-2-\frac{1}{2}=\frac{1}{2}$ charges after (R5) to (R6). Second, when $m_2=2$, it follows that $m_3\leq 1$. If $m_3=0$, then $f$ is adjacent to at most three ($4^-,4^-,5^+$)-faces (Note that when $f$ is adjacent to a ($4^-,4^-,5^+$)-face, $f$ has to be incident with a 4-vertex. So, $f$ is incident with at most four 6-vertices in this case).
Therefore, $f$ has at least $7-4-2-3\times \frac{1}{6}=\frac{1}{2}$ charges after (R5) to (R6).
If  $m_3=1$, then $f$ is not adjacent to any ($4^-,4^-,5^+$)-face. Therefore, $f$ has at least $7-4-2-\frac{1}{2}=\frac{1}{2}$ charges after (R5) to (R6). Third, when $m_2\leq 1$, it has that $m_3\leq 2$. In this case, we can  see that  $f$ has at least $7-4-1-\frac{1}{2}-\frac{1}{2}=1$ charges after (R5) to (R6). All of the above show that $f$ sends $v$ at least $\frac{1}{4}\times \frac{1}{2}= \frac{1}{8}$ by (R7).
\end{proof}

By Lemma \ref{chargeto6-vertex}, we can see that $c'(f)\geq 0$ for every $7^+$-face $f$.

\subsection{Final charge of vertices}

We start with an observation and a lemma.

\textbf{Observation.} Let $v$ be a vertex of $H$. Since $H$ has no subgraph isomorphic to a diamond, we have  $n_3\leq \lceil\frac{d_H(v)-1}{2}\rceil$. Moreover, if $v$ is a 6-vertex, then by the condition of Theorem \ref{mainresult}, $n_4\leq 3$ and $n_3+n_4\leq 3$.

\begin{lemma} \label{chargeto6-vertex1}
Suppose that $v$ is a 6-vertex incident with three consecutive faces of size 4, 6 and 4, respectively, where the 6-face is denoted by $f$; See Figure \ref{figureadd-new} (a). Then by $(R7)$, $f$ gives $v$ at least

$(1)$ $\frac{1}{8}$, if $f$ is incident with at most two $3^-$-vertices;
\vspace{0.1cm}

$(2)$ $\frac{1}{6}$, if $f$ is incident with three $3^-$-vertices and $d_H(v_1)=d_H(v_2)=2$ (See Figure \ref{figureadd-new} (b));
\vspace{0.1cm}

$(3)$ $\frac{1}{9}$, if $f$ is incident with three $3^-$-vertices and $d_H(v_1)=d_H(v_2)=3$ (See Figure \ref{figureadd-new} (c));
\vspace{0.1cm}

$(4)$ $\frac{2}{9}$, if $f$ is incident with three $3^-$-vertices and $d_H(v_1)\neq d_H(v_2)$.

%\Figure[t!](topskip=0pt, botskip=0pt, midskip=0pt)[width=8.5cm]{fignew.pdf}
%{Cases of Lemma \ref{chargeto6-vertex1} \label{figureadd-new}}

\begin{figure}[H]
\centering
  \includegraphics[width=300pt]{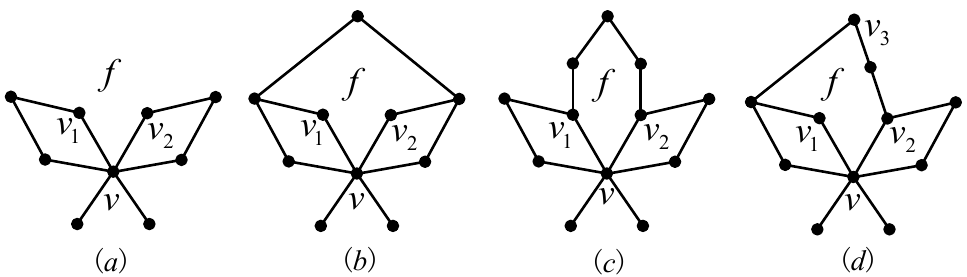}
  \caption{Cases of Lemma \ref{chargeto6-vertex1} }\label{figureadd-new}
\end{figure}

\end{lemma}
\begin{proof}
Since $H$ contains no subgraph isomorphic to a diamond or a house, $v_1$ and $v_2$ are not incident with a 3-face if $d_H(v_1)\leq 3$ and $d_H(v_1)\leq 3$, and  $f$ is incident with at most one 2-vertex that is incident with a 3-face.

For (1), if  $d_H(v_1)\leq 3$ and $d_H(v_2)\leq 3$, then  $f$ is  adjacent to at most two ($4^-,4^-,5^+$)-faces. Therefore, $f$ has at least $6-4-\frac{1}{2}-\frac{1}{2}-2\times \frac{1}{6}=\frac{4}{6}$ charges after (R5) to (R6), and $v$ receives at least $\frac{1}{4}\times \frac{4}{6}=\frac{1}{6}$ from $f$.

If  exact one of $v_1$ and $v_2$ is a $3^-$-vertex, say $v_1$, then we consider two cases. First, $m_2+m_3=2$. In this case, $f$ is  adjacent to at most one ($4^-,4^-,5^+$)-face. Particular, if $v_1$ is a 2-vertex, denoted by $x$=$N_H(v_1)\setminus \{v\}$, then $x$ is not adjacent to another 2-vertex that is incident with a 3-face by Lemma \ref{lemma2}. This implies that when $f$ is incident with a 2-vertex that is incident with a 3-face, the 2-vertex is a neighbor of $v_2$ and so $f$ is not incident with ant ($4^-,4^-,5^+$)-face. Consequently, $f$ has at least
$\min \{6-4-\frac{1}{2}-1=\frac{1}{2}$ ($v_1$ is a 2-vertex), $6-4-\frac{1}{3}-1-\frac{1}{6}=\frac{1}{2}$ ($v_1$ is a 3-vertex)\}=$\frac{1}{2}$ charges after (R5) to (R6), and $v$ receives at least $\frac{1}{4}\times \frac{1}{2}=\frac{1}{8}$ from $f$. Second,  $m_2+m_3=1$, i.e. $f$ is  incident with only one $3^-$-vertex $v_1$. Then,  $f$ is adjacent to at most three ($4^-,4^-,5^+$)-faces.
Therefore, $f$ has at least $6-4-\frac{1}{2}-3\times \frac{1}{6}=1$ charges after (R5) to (R6), and $v$ receives at least $\frac{1}{4}$ from $f$.

If  $d_H(v_1)\geq 4$ and $d_H(v_2)\geq 4$, then $m_2+m_3\leq 2$. When  $m_2+m_3=2$, one can readily check that  $f$ is not adjacent to any ($4^-,4^-,5^+$)-face. Therefore, $f$ has at least $6-4-\frac{1}{2}-1=\frac{1}{2}$ charges after (R5) to (R6), and $v$ receives at least $\frac{1}{4}\times \frac{1}{2}=\frac{1}{8}$ from $f$.
When $m_2+m_3=1$, it has that  $f$ is adjacent to at most two ($4^-,4^-,5^+$)-faces. Therefore, $f$ has at least $6-4-1-2\times \frac{1}{6}=\frac{2}{3}$ charges after (R5) to (R6), and $v$ receives at least $\frac{1}{5}\times \frac{2}{3}=\frac{2}{15}$ from $f$.  When $m_2+m_3=0$, it has that  $f$ has at least $6-4-4\times \frac{1}{6}=\frac{4}{3}$ charges after (R5) to (R6), and $v$ receives more than  $\frac{1}{6}\times \frac{4}{3}=\frac{2}{9}$ from $f$.

For (2) and (3), it follows that $f$ is not adjacent to any ($4^-,4^-,5^+$)-face. If $d_H(v_1)=d_H(v_2)=2$, then the other $3^-$-vertex incident with $f$ is a 3-vertex. Therefore, $f$ has at least $6-4-\frac{1}{2}-\frac{1}{2}-\frac{1}{2}=\frac{1}{2}$ charges  after (R5) to (R6), and $v$ receives at least $\frac{1}{2}\times \frac{1}{3}= \frac{1}{6}$ from $f$. If $d_H(v_1)=d_H(v_2)=3$, then $f$ has at least $6-4-\frac{1}{3}-\frac{1}{3}-1=\frac{1}{3}$ charges after (R5) to (R6),  and $v$ receives at least $\frac{1}{3}\times \frac{1}{3}= \frac{1}{9}$ from $f$.

For (4), it is clear that $v_1$ and $v_2$ are  $3^-$-vertices. Without loss of generality, we assume $d_H(v_1)=2$ and $d_H(v_2)=3$; See Figure \ref{figureadd-new} (d), where $v_3$ is another $3^-$-vertex incident with $f$.  If $d_H(v_3)=2$, then $v_3$ is not incident with a 3-face by Lemma \ref{lemma2}. Therefore, $f$ has at least $6-4-\frac{1}{2}-\frac{1}{3}-\frac{1}{2}=\frac{2}{3}$ charges after (R5) to (R6), and $v$ receives at least $\frac{1}{3}\times \frac{2}{3}=\frac{2}{9}$ from $f$.
\end{proof}

In the following, we turn to the proof of $c'(v)\geq 0$ for every $v\in V(H)$. Let $v\in V(H)$ be a vertex of $H$. By Lemma \ref{lemma1} (1), we have $d_H(v)\geq 2$.

Suppose that $v$ is a 2-vertex. Then $v$ has two neighbors with degree 6 by Lemma \ref{lemma1} (1). If $v$ is incident with a 3-face, then $v$ is incident with a $5^+$-face. So, $v$  receives $\frac{1}{2}$ from each of its neighbors by (R3), and receives 1 from its incident $5^+$-face. Hence, $c'(v)=2-4+2\times \frac{1}{2}+1=0$. If $v$ is not incident with a 3-face, then by (R3) $v$  receives $\frac{4}{5}$ from its master and $\frac{1}{5}$ from its other neighbor of degree 6, and receives $\frac{1}{2}$ from each of its adjacent $4^+$-face by (R4) and (R5). Therefore, $c'(v)=2-4+\frac{4}{5}+\frac{1}{5}+2\times \frac{1}{2}=0$.

Suppose that $v$ is a 3-vertex. If $v$ is incident with a 3-face, then $v$ is incident with two $5^+$-faces since $H$ does not contain any subgraph isomorphic to a house. So, by $(R5)$ $c'(v)=3-4+2\times \frac{1}{2}=0$. If $v$ is not incident with any 3-face, then  $v$ is incident with three $4^+$-faces. Hence, by (R4) and (R5), $c'(v)=-1+3\times \frac{1}{3}=0$.

Suppose that $v$ is a 4-vertex. By the discharging rules (R1) to (R7), we have $c'(v)=c(v)=0$.

Suppose that $v$ is a 5-vertex. By the observation, we have $n_3\leq 2$.
If $n_3=0$, then $v$ is incident with at most five 4-faces. So, by (R1), $c'(v)\geq 1-5\times \frac{1}{5}$=0.
If  $n_3=1$, then $n_4\leq 2$ by the condition of Theorem \ref{mainresult}. So, by (R1), $c'(v)\geq 1-\frac{1}{2}-2\times\frac{1}{5}$=$\frac{1}{10}$. If $n_3$=2, then $n_4=0$ by the same reason. So, by (R1)  $c'(v)\geq 1-2\times \frac{1}{2}$=$0$.

Suppose that $d_H(v)=6$. By the observation we have $n_3$+$n_4\leq 3$. Denote by $t_2$  the number of 2-vertices adjacent to $v$. Then $t_2\leq 5$ by  Lemma \ref{lemmaadd}. When $t_2=0$, it is clear that $c'(v)\geq 6-4-3\times \frac{1}{2}=\frac{1}{2}$
by (R1) and (R2). When $t_2=1$, denote by $u$ the unique 2-vertex adjacent to $v$.  First, $n_3+n_4\leq 2$. Then, $c'(v)\geq 6-4-\frac{4}{5}-2\times \frac{1}{2}=\frac{1}{5}$ by (R1), (R2) and (R3). Second, $n_3+n_4=3$. In this case, by the condition of Theorem \ref{mainresult}, either $n_3=3$ or $n_4=3$. For the former, we have $c'(v)\geq 6-4-\frac{1}{2}-3\times \frac{1}{2}=0$ by (R1),(R2) and (R3). For the latter case, if $u$ is incident with a $(2,6,3,6)$-face, then $v$ is incident with at most one (3,5,3,6)-face by Lemma \ref{lemmaadd}. Therefore, $c'(v)\geq 6-4-\frac{4}{5}- \frac{5}{12}-\frac{7}{15}-\frac{1}{6}=\frac{3}{20}$ by (R2) and (R3); If $u$ is not incident with a $(2,6,3,6)$-face, then $v$ is incident with at most two (3,5,3,6)-faces. Therefore, $c'(v)\geq 6-4-\frac{4}{5}-\frac{1}{4}-\frac{7}{15}-\frac{7}{15}=\frac{1}{60}$ by (R2) and (R3). In what follows, we assume $t_2\geq 2$, and then by Lemma \ref{lemma2}, we have that every 2-vertex is not incident with a 3-face. Thus, when $n_3$+$n_4=0$, $c'(v)\geq 6-4-\frac{4}{5}-4\times \frac{1}{5}=\frac{2}{5}$ by (R3). Now, we further consider the following three cases.

\textbf{Case 1.} $n_3$+$n_4=1$. If $n_3=1$ and $n_4=0$, then $t_2\leq 4$ by Lemma \ref{lemma2}.  Therefore, $c'(v)\geq 6-4-\frac{4}{5}-3\times \frac{1}{5}-\frac{1}{2}=\frac{1}{10}$ by (R1) and (R3). If $n_3=0$ and $n_4=1$,  denoted by $f$ the 4-face incident with $v$, then by  Lemma \ref{lemmaadd} and (R2),(R3), when $f$ is a (2,6,3,6)-face, it has that $t_2\leq 4$  and $c'(v)\geq  6-4-\frac{4}{5}-3\times \frac{1}{5}-\frac{5}{12}=\frac{11}{60}$; When $f$ is a ($2,6,4^+,6$)-face,  $t_2\leq 5$ and $c'(v)\geq  6-4-\frac{4}{5}-4\times \frac{1}{5}-\frac{1}{4}=\frac{3}{20}$;  When $f$ is a ($3,5^+,3,5^+$)-face, $t_2\leq 3$ and $c'(v)\geq  6-4-\frac{4}{5}-2\times \frac{1}{5}-\frac{7}{15}=\frac{1}{3}$; When $f$ is a ($3,5^+,4^+,5^+$)-face, $t_2\leq 4$  and $c'(v)\geq  6-4-\frac{4}{5}-3\times \frac{1}{5}-\frac{1}{6}=\frac{13}{30}$;  When $f$ is a ($4^+,4^+,4^+,4^+$)-face, $t_2\leq 4$  and $c'(v)\geq  6-4-\frac{4}{5}-3\times \frac{1}{5}=\frac{3}{5}$.

\textbf{Case 2.} $n_3$+$n_4=2$. If $t_2=2$, then $c'(v)\geq  6-4-\frac{4}{5}-\frac{1}{5}-\frac{1}{2}-\frac{1}{2}=0$ by (R1), (R2) and (R3). If $t_2\geq 3$, then  we have $n_3\leq 1$ by Lemma \ref{lemma2} and the assumption that $H$ contains no subgraph isomorphic to a diamond. In the following, we consider two subcases.

%\Figure[t!](topskip=0pt, botskip=0pt, midskip=0pt)[width=5.5cm]{fignewx.pdf}
%{Illustration for the proof of Case 2 \label{figurex1}}

\begin{figure}[H]
\centering
  \includegraphics[width=180pt]{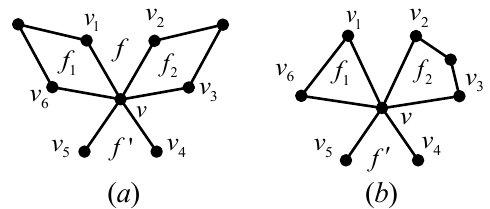}
  \caption{Illustration for the proof of Case 2}\label{figurex1}
\end{figure}

\textbf{Case 2.1.}  $n_3=0,n_4=2$. Let $N_H(v)=\{v_1,v_2,v_3,v_4,v_5,v_6\}$, and $f_1$ and $f_2$ be the two 4-faces incident with $v$, See Figure \ref{figurex1} (a). By Lemma \ref{old3} each of $f_1$ and $f_2$ is incident with at most one 2-vertex, so $t_2\leq 4$.

When $t_2=3$, it has that at least one of $f_1$ and $f_2$ is incident with a 2-vertex, say $f_1$.  By Lemma \ref{lemma1} (1), it follows that $f_1$ is a (2,6,$3^+$,6)-face. If $f_1$ is a (2,6,$4^+$,6)-face, then $c'(v)\geq  6-4-\frac{4}{5}-2\times \frac{1}{5}-\frac{1}{4}-\frac{7}{15}=\frac{1}{12}$ by (R2) and (R3). If $f_1$ is a (2,6,$3$,6)-face, then we further consider the following two cases regarding to $f_2$. First, $f_2$ is incident with a 2-vertex, i.e. $f_2$ is a (2,6,$3^+$,6)-face by Lemma \ref{lemma1} (1). If $f_2$ is a (2,6,$4^+$,6)-face, then $c'(v)\geq  6-4-\frac{4}{5}-2\times \frac{1}{5}-\frac{1}{4}-\frac{5}{12}=\frac{2}{15}$ by (R2) and (R3). Otherwise,  $f_2$ is a (2,6,$3$,6)-face. In this case, according to Lemma \ref{lemma-add} (1) and (2), we can deduce that $f$ is a $6^+$-face. Therefore, according to Lemmas \ref{chargeto6-vertex} and \ref{chargeto6-vertex1} $v$ receives at least $\frac{1}{8}$ from $f$ by (R7). So $c'(v)\geq  6-4-\frac{4}{5}-2\times \frac{1}{5}-2\times\frac{5}{12}+\frac{1}{8}=\frac{11}{120}$ by (R2), (R3) and (R7). Second, $f_2$ is not incident with any 2-vertex. Then, we consider $m_3(f_2)$, the number of 3-vertices incident with $f_2$. Obviously, $m_3(f_2)\leq 1$ by Lemma \ref{lemma1} (1) and Lemma \ref{lemmaadd} (Since $t_2$=3). Hence,  $c'(v)\geq  6-4-\frac{4}{5}-2\times \frac{1}{5}-\frac{5}{12}-\frac{1}{6}=\frac{13}{60}$ by (R2) and (R3).

When $t_2=4$, it follows that both $f_1$ and $f_2$ are ($2,6,3^+,6$)-faces by Lemma \ref{lemma1} (1) and Lemma \ref{old3}, and at most one of them is a ($2,6,3,6$)-face by Lemma \ref{lemmaadd}.
Naturally, $v_4$ and $v_5$ are 2-vertex, and neither of them is incident with a 3-face. Denote by the face incident with $v,v_1,v_2$ by $f$; see Figure \ref{figurex1} (a). Then, $d_H(f')\geq 6$ by Lemma \ref{lemma-add} (1). If $d_H(f')=6$, then by (R7) it sends at least $\min \{\frac{1}{3}\times (6-4-2\times\frac{1}{2}-\frac{1}{2}), \frac{1}{4}\times (6-4-2\times\frac{1}{2})\}=\frac{1}{6}$ to $v$. If $d_H(f')\geq 7$, then it sends at least $\frac{1}{8}$ to $v$ by Lemma \ref{chargeto6-vertex}. Therefore,  $c'(v)\geq  6-4-\frac{4}{5}-3\times \frac{1}{5}-\frac{5}{12}-\frac{1}{4}+\frac{1}{8}=\frac{7}{120}$ by (R2) and (R3).

 \textbf{Case 2.2.}  $n_3=1,n_4=1$.  Denote by $f_1$ and $f_2$ the 3-face and 4-face incident with $v$, and $f'$  the face incident with $v$ and not adjacent to $f_1$ or $f_2$; See Figure \ref{figurex1} (b).  By Lemmas \ref{old3} and \ref{lemma2}, we can see that  $f_1$ is not incident with any 2-vertex (Since $t_2\geq 3$), $f_2$ is  incident with one 2-vertex, and $v_4,v_5$ are 2-vertices. Obviously, $v_4$ and $v_5$ are not incident with any 3-face. Additionally, by Lemma \ref{lemma-add} (1) we can see that $d_H(f')\geq 6$. Thus, with an analogous proof as above (Case 2.1),  $v$ can receive at least $\frac{1}{8}$ from $f'$ by (R7).  Therefore, $c'(v)\geq  6-4-\frac{4}{5}-2\times \frac{1}{5}-\frac{1}{2}-\frac{5}{12}+\frac{1}{8}=\frac{1}{120}$ by (R2) and (R3).

\textbf{Case 3.} $n_3$+$n_4=3$. In this case, since we assume $t_2\geq 2$, it has that $n_3=0$ and $n_4=3$ by the condition of Theorem \ref{mainresult} and Lemma \ref{lemma2}. Denote by $f_1,f_2,f_3$ the three 4-faces incident with $v$, and $f'_1,f'_2,f'_3$ the three $5^+$-faces incident with $v$; See Figure \ref{figurex2}. By Lemma \ref{old3} it has that $t_2\leq 3$.

%\Figure[t!](topskip=0pt, botskip=0pt, midskip=0pt)[width=3.2cm]{fignewx3.pdf}
%{Illustration for the proof of Case 3 \label{figurex2}}

\begin{figure}[H]
\centering
  \includegraphics[width=100pt]{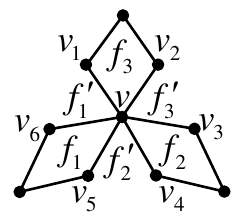}
 \caption{Illustration for the proof of Case 3}\label{figurex2}
\end{figure}

\textbf{Case 3.1.} $t_2=2$. Without loss of generality, we assume $f_1$, $f_2$ are incident with  2-vertices. Then, both of $f_1$ and $f_2$ are ($2,6,3^+,6$)-faces by Lemma \ref{lemma1} (1). %If one of $f_1$ and $f_2$ is  a (2,6,$4^+$,6)-face, then $c'(v)\geq  6-4-\frac{4}{5}- \frac{1}{5}-\frac{1}{2}-\frac{5}{12}+\frac{1}{8}=\frac{1}{120}$ by (R2) and (R3).
Now, we turn to considering $f_3$. First, $f_3$ is incident with two 3-vertices, i.e. both $v_1$ and $v_2$ are 3-vertices. Then, at most one of $f_1$ and $f_2$ is a (2,6,3,6)-face. If both $f_1$ and $f_2$ are (2,6,$4^+$,6)-face, then $c'(v)\geq  6-4-\frac{4}{5}- \frac{1}{5}-\frac{7}{15}-\frac{1}{4}-\frac{1}{4}=\frac{1}{30}$ by (R2) and (R3). Otherwise, we assume $f_1$ is a (2,6,3,6)-face.
In this case, by (R7), when $d_H(v_5)=3$,   according to Lemmas \ref{chargeto6-vertex} and \ref{chargeto6-vertex1} $v$ can receive at least $\frac{1}{3}\times (5-4-\frac{1}{3}-\frac{1}{2})=\frac{1}{18}$ ( When $f'_i$ is a 5-face) from $f'_i$ for $i=1,2,3$; When  $d_H(v_5)=2$,  we have that $d_H(v_6)=3$, and $v$ can receive at least $\frac{1}{3}\times (5-4-\frac{1}{3}-\frac{1}{3})=\frac{1}{9}$ (When $f'_1$ is a 5-face) from $f'_1$ and at least $\frac{1}{3}\times (5-4-\frac{1}{3}-\frac{1}{2})=\frac{1}{18}$ (When $f'_3$ is a 5-face) from $f'_3$. Additionally,  by  Lemmas \ref{chargeto6-vertex} and \ref{chargeto6-vertex1} $v$ can receive at least $\frac{1}{9}$ from $f'_i$ for $i=1,2,3$ when $f'_i$ is a $6^+$-face. Therefore, $c'(v)\geq  6-4-\frac{4}{5}- \frac{1}{5}-\frac{7}{15}-\frac{5}{12}-\frac{1}{4}+3\times \frac{1}{18}=\frac{1}{30}$ by (R2) and (R3). Second, $f_3$ is incident with at most one 3-vertices. Then, $c'(v)\geq  6-4-\frac{4}{5}- \frac{1}{5}-\frac{1}{6}-\frac{5}{12}-\frac{5}{12}=0$ by (R2) and (R3).

\textbf{Case 3.2.} $t_2=3$. Then, each of $f_i$, $i\in \{1,2,3\}$ is  a (2,6,$3^+$,6)-face. Particularly, by Lemma \ref{lemmaadd} at most two of them are (2,6,3,6)-faces. When $v$ is incident with at most one (2,6,3,6)-face, we without loss of generality assume that $f_3$ is a (2,6,3,6)-face, and by symmetry let $d_H(v_1)=3$ and $d_H(v_2)=2$. If $v_6$ (or $v_3$) is a 2-vertex, then by Lemma \ref{lemma-add} (2) (or Lemma \ref{lemma-add} (1)) $f'_1$ (or $f'_3$) is a $6^+$-face. If neither  $v_6$ nor $v_3$ is a 2-vertex, then $v_4$ and $v_5$ are 2-vertices and by Lemma \ref{lemma-add} (1) $f'_2$  is a $6^+$-face. Therefore, by Lemmas \ref{chargeto6-vertex} and  \ref{chargeto6-vertex1} $v$ can receives at least $\frac{1}{8}$ from $f'_i$ for some $i\in \{1,2,3\}$ by (R7) (Note that $v$ is adjacent to at most one 3-vertex in this case). Hence, $c'(v)\geq  6-4-\frac{4}{5}- 2\times\frac{1}{5}-\frac{1}{4}-\frac{5}{12}-\frac{1}{4}+\frac{1}{8}=\frac{1}{120}$ by (R2) and (R3). When $v$ is incident with two (2,6,3,6)-faces,  by symmetry we assume $f_1$ and $f_2$ are (2,6,3,6)-faces.
Without loss of generality, we assume $d_H(v_1)=2$, and then $d_H(v_2)\geq 4$.

First, if $d_H(v_6)=2$ and $d_H(v_5)=3$, then $f'_1$ is a $6^+$-face by Lemma \ref{lemma-add} (1). When $f'_1$ is a 6-face, $f'_1$ gives $v$ at least $\frac{1}{3}\times (6-4-\frac{1}{2}-\frac{1}{2}-\frac{1}{2})=\frac{1}{6}$; When $f'_1$ is a $k$-face ($k\geq 7$), by $(R2), (R3)$ and $(R7)$ we can deduce that $f'_1$ gives $v$ at least $\frac{k-4-2\times \frac{1}{2}-(m_2(f'_1)-2)-\frac{1}{2}\times m_3(f'_1)-\frac{1}{6}\times (k-2m_2(f'_1)-2m_3(f'_1))}{k-m_2(f'_1)-m_3(f'_1)}\geq \frac{5}{24} (k=7, m_2(f'_1)=3, m_3(f'_1)=0)$.

Consider the face $f'_2$. If $f'_2$ is a $6^+$-face, then  by $R 7$ $f'_2$ sends at least $\frac{1}{9}$ to $v$ according to Lemmas \ref{chargeto6-vertex} and \ref{chargeto6-vertex1}. If $f'_2$ is a 5-face, then by Lemmas \ref{lemma-add} (1) and \ref{lemma-add} (2), we can deduce that $d_H(v_4)=3$. So, $f'_2$ sends at least $\frac{1}{3}\times (5-4-\frac{1}{3}-\frac{1}{3})=\frac{1}{9}$ to $v$ by (R7).

Consider the face $f'_3$. If  $f'_3$ is a $6^+$-face, then  by $R 7$ $f'_3$ sends at least $\frac{1}{8}$ to $v$ according to Lemmas \ref{chargeto6-vertex} and \ref{chargeto6-vertex1}. If $f'_3$ is a $5$-face, then when $d_H(v_3)=3$, $f'_3$ sends at least $\frac{1}{3}\times (5-4-\frac{1}{3}-\frac{1}{2})=\frac{1}{18}$ to $v$; When $d_H(v_3)=2$, let $v_2,v,v_3,u_1,u_2$ be the five vertices incident with $f'_3$, where $u_1$ is adjacent to $v_3$ and $u_2$ is adjacent to $v_2$. Then, $d_H(u_1)=6$ by Lemma \ref{lemma1} (1), and $d_H(u_2)\geq 3$ by Lemma \ref{lemma-add} (3). If $d_H(u_2)\geq 4$, then $v$ can receive at least $\frac{1}{3}\times (5-4-\frac{1}{2}-2\times \frac{1}{6})=\frac{1}{18}$ from $f'_3$. Otherwise, if  $d_H(u_2)=3$, then $d_H(v_2)\geq 5$, and $f_3$ is incident with only one $3^-$-vertex $v_1$. So $f_3$ has at least $2\times \frac{1}{4}+\frac{1}{5}-\frac{1}{2}=\frac{1}{5}$ after (R1), (R2) and (R4). So, $f_3$ gives $v$ at least $\frac{1}{15}$ by (R7) in this case.

To sum up, we have  $c'(v)\geq  6-4-\frac{4}{5}- 2\times \frac{1}{5}-\frac{1}{4}-\frac{5}{12}-\frac{5}{12}+\frac{1}{6}+\frac{1}{9}+
\frac{1}{18}=\frac{1}{20}$ by (R2) and (R3).

Second, if $d_H(v_6)=3$ and $d_H(v_5)=2$, then both $f'_1$ and $f'_2$ are $6^+$-faces by Lemma \ref{lemma-add} (1) and (2). So, each of $f'_1$ and $f'_2$ gives $v$ at least $\frac{1}{8}$ by Lemmas \ref{chargeto6-vertex} and \ref{chargeto6-vertex1}. In addition, with the similar analysis as the above, $v$ can receive at least either $\frac{1}{18}$ from $f'_3$ or $\frac{1}{15}$ from $f_3$. Hence,  $c'(v)\geq  6-4-\frac{4}{5}- 2\times \frac{1}{5}-\frac{1}{4}-\frac{5}{12}-\frac{5}{12}+\frac{1}{8}+\frac{1}{8}+
\frac{1}{18}=\frac{1}{45}$ by (R2) and (R3).

In all cases we have shown that $c'(x)\geq 0$ for every $x\in V(H)\cup F(H)$. Therefore, $\sum_{x\in V(H)\cup F(H)}c(x)=\sum_{x\in V(H)\cup F(H)}c'(x)\geq 0$, a contradiction. This completes the  proof of Theorem \ref{mainresult}.

\section{Conclusion}
By using the ``discharging" approach, we obtain a sufficient condition for a planar graph of maximum degree 6 to be 7-colorable. Since no  planar graphs of maximum degree 6 that are not totally 7-colorable are found, it is widely believed  that every planar graph of maximum degree 6 has a total 7-coloring \cite{Shen2009a}.
Our result enhance the reliability of this conjecture. Nevertheless, to prove TCC for planar graphs, it still requires  persistent efforts on the study of structures of planar graphs of maximum degree 6. As a future work, we would like to further explore the structural properties of this kind of graphs, as well as the possibility of applying them in the proof of TCC for planar graphs.

\acknowledgements
\label{sec:ack}
This work was supported in part by the National Natural Science Foundation of China under Grants 61672051, 61872101.

\nocite{*}
%\bibliographystyle{abbrvnat}
% use the following instead if you encounter problems
%\bibliographystyle{alpha}
%\bibliography{refs}
%\label{sec:biblio}

\end{document}